\documentclass[11pt]{amsart}
   \usepackage[margin=3.7cm]{geometry}

\numberwithin{equation}{section}
\usepackage{amsmath,amsfonts,amsthm,amssymb,amscd,mathabx, verbatim,graphicx,color,multirow,booktabs, caption,tikz,tikz-cd, mathdots,bm}
\usepackage{tikz-cd}
 \usepackage[pagebackref]{hyperref} 
\usetikzlibrary{positioning}

\newtheorem{theorem}{Theorem}[section]
\newtheorem{lemma}[theorem]{Lemma}

\newtheorem{proposition}[theorem]{Proposition}
 \newtheorem{corollary}[theorem]{Corollary}

      \newtheorem{conjecture}{Conjecture}

      \theoremstyle{definition}
     
     \newtheorem{example}[theorem]{Example}
     
     \theoremstyle{remark}
     \newtheorem{remark}[theorem]{Remark}

\newcommand{\Syl}{\mathop{\mathrm{Syl}}}

 \definecolor{mycolor}{rgb}{0.55,0.0,0.16}
  \definecolor{myred}{rgb}{0.75,0.0,0.16} 
  \definecolor{mygreen}{rgb}{0.0,0.4,0.16} 
  \definecolor{myviolet}{rgb}{1,0,1} 
   \definecolor{mypink}{rgb}{0.67,0,0.47}

\hypersetup{
colorlinks=true,
linkcolor=true,
linktocpage=true,
pageanchor=true,
hyperindex=true
}

 \AtBeginDocument{
     \hypersetup{
  linkcolor=mycolor,
  urlcolor=mypink,
citecolor=mygreen
}
     }

\makeatletter
\@namedef{subjclassname@2020}{%
  \textup{2020} Mathematics Subject Classification}
\makeatother

\subjclass[2020]{Primary: 20D20, 20Fxx, 20D05}
\keywords{Sylow subgroups, nilpotent subgroups, simple groups, synchronization.} 

\author[Hong Yi Huang]{Hong Yi Huang}
\address{\parbox{\linewidth}{Hong Yi Huang, Alfréd Rényi Institute of Mathematics, Budapest, Reáltanoda u. 13-15, H-1053, Hungary}}
\email{11612012@mail.sustech.edu.cn} 

\author[Francesca Lisi]{Francesca Lisi}
\address{\parbox{\linewidth}{Francesca Lisi}}
\email{francesca.lisi@postecert.it}

\author[Aluna Rizzoli]{Aluna Rizzoli}
\address{\parbox{\linewidth}{Aluna Rizzoli, Department of Mathematics, King's College London, Strand, London WC2R 2LS, United Kingdom}}
\address{\parbox{\linewidth}{Heilbronn Institute for Mathematical Research, Bristol, United Kingdom}}
\email{aluna.rizzoli@kcl.ac.uk} 

\author[Luca Sabatini]{Luca Sabatini}
\address{\parbox{\linewidth}{Luca Sabatini, Mathematics Institute, University of Warwick}}
\email{luca.sabatini@warwick.ac.uk, sabatini.math@gmail.com} 

 \title[Sylow synchronization in finite groups]{Sylow synchronization in finite groups:\\the good case} 

\begin{document} 

\maketitle 

\begin{abstract} 
Let $G$ be a finite solvable group such that,
for any prime $p$ and any quotient $Q$ of $G$,
there are two Sylow $p$-subgroups of $Q$ intersecting in $O_p(Q)$.
Then, for every family $(P_i)_{i=1}^n$ of Sylow subgroups of $G$ for distinct primes $p_1,\ldots,p_n$, there exists $x \in G$ such that $P_i \cap P_i^x = O_{p_i}(G)$ for all $i$.
This covers groups of odd order,
partially settling a conjecture of the second and fourth authors
and unifying old results of Bialostocki and Mann on the intersection of nilpotent subgroups.

We also prove a result for all symmetric and alternating groups,
completing the proof of the conjecture for simple groups
as initiated by Burness and the first author.
\end{abstract}


\vspace{0.2cm}
\section{Introduction} \label{sec1}

In \cite{LS26}, the second and fourth authors proposed
the following ``synchronization conjecture'' for Sylow intersections,
which also appears as Question~21.26 in the Kourovka Notebook \cite{Kou26}: 

\begin{conjecture}[\cite{LS26}] \label{conjMain}
Let $G$ be a finite group and let $(P_i)_{i=1}^n$ be Sylow subgroups for distinct primes.
Then there exists $x \in G$ such that $P_i \cap P_i^x$ is inclusion-minimal in $\{ P_i \cap P_i^g : g \in G\}$,
for all $i=1,\ldots,n$.
\end{conjecture}

This was motivated by problems concerning the intersection of nilpotent subgroups.
In \cite{LS26}, Conjecture~\ref{conjMain} was proved
for large alternating groups and for metanilpotent groups of odd order.
In \cite{BH26}, Burness and the first author combined probabilistic and computational ideas
to prove Conjecture~\ref{conjMain} for all non-alternating simple groups,
obtaining, as a consequence,
the proof of a striking conjecture of Vdovin \cite[Question~15.40]{Kou26}.

For a prime $p$, the $p$-core $O_p(G)$ is the intersection of all the Sylow $p$-subgroups of $G$.
Following \cite{LS26}, we say that a finite group $G$ has $(*)_p$
if there exist two Sylow $p$-subgroups intersecting in $O_p(G)$,
and that $G$ has $(*)$ if it has $(*)_p$ for all $p$.
Since $(*)_p$ does not pass to quotients in general (see Example~\ref{ex:Q}),
it is natural to refer to the ``good case''
as to the situation where all quotients of $G$ have $(*)$;
this is in contrast to \cite[Conjecture~B]{LS26}, where $(*)$ was required only for $G$.
This property removes the burden of dealing with annoying minimal intersections.

The main contribution of this paper is to handle the good case
under the further assumption that the group is solvable.

\begin{theorem} \label{thMain}
Let $G$ be a finite solvable group such that every quotient of $G$ has $(*)$.
Then Conjecture~\ref{conjMain} holds for $G$.
\end{theorem}

By the Feit--Thompson theorem \cite{FT63} and a result of It\^o \cite{Ito58},
the groups of odd order satisfy the hypotheses of Theorem~\ref{thMain}.
Hence we have

\begin{corollary} \label{corOdd}
Conjecture~\ref{conjMain} holds for groups of odd order.
\end{corollary}

This was reported as an important open case in both \cite{LS26} and \cite{BH26}.
As a consequence, we also obtain a positive solution to \cite[Conjecture~C]{LS26}.
We write $F(G)$ for the Fitting subgroup of $G$.

\begin{corollary} \label{corBiaMan}
Let $G$ be a group of odd order and let $H$ be a nilpotent subgroup of $G$.
Then there exists $x \in G$ such that $H \cap H^x \subseteq F(G)$.
\end{corollary}

The case where $H$ is a Hall subgroup was proven by Bialostocki \cite[Theorem~B.4]{Bia75},
while the case where $H$ contains $F(G)$ was proven by Mann \cite[page~263]{Man75}.
In fact, it was the search for a proof of Corollary~\ref{corBiaMan}
that originally led to state Conjecture~\ref{conjMain}.

In brief, the proof of Theorem~\ref{thMain} works as follows.
Choose a prime $p$ such that $O_p(G) \neq 1$, let $A=O_p(G)$ and $H/A = F(G/A)$.
We handle $G/A$ by induction, obtaining a coset $xA$ which allows us to enter in $H/A$.
Up to quotienting by the Frattini subgroup, we can look at $A$ as an $H/A$-module.
We then apply a strong synchronization result
for nilpotent linear groups (Proposition~\ref{prop:translated}),
which provides an element
$a \in A$ that deals with the primes different from $p$.
Since $H/A$ is a $p'$-group and $a$ fixes the Sylow $p$-subgroups by conjugation,
we will use the product $xa$ as the desired element dealing with all primes simultaneously.

In Section~\ref{sec3}, we make some effort to settle the case of small symmetric and alternating groups, as prompted in \cite[Remark~7]{BH26}.
Combining this result with \cite[Theorem~A]{BH26}, we obtain

\begin{theorem} \label{thSimple}
Conjecture~\ref{conjMain} holds for finite simple groups.
\end{theorem}

We recall that all finite simple groups have $(*)$,
according to a theorem of Mazurov and Zenkov \cite{MZ96}.

\vspace{0.1cm}
\section{Solvable groups} \label{sec2}

\subsection{Hereditary properties of $(*)$}

\begin{lemma} \label{lem:normal}
If $G$ has $(*)_p$ and $N \unlhd G$, then $N$ has $(*)_p$.
\end{lemma}
\begin{proof}
Choose $P \in \Syl_p(G)$ and $g \in G$ such that $P \cap P^g = O_p(G)$.
Then the Sylow $p$-subgroups $P \cap N$ and $P^g \cap N$ of $N$
intersect in $O_p(G) \cap N = O_p(N)$.
\end{proof}

However, $(*)_p$ does not pass to quotients.

\begin{example} \label{ex:Q}
Let $V = (\mathbb{F}_3)^2$ and let $P \cong D_8$ be
the group of signed $2 \times 2$ permutation matrices in $\mathbb{F}_3$.
It is easy to check that $H =V \rtimes P$ does not have $(*)_2$.
On the other hand, the group algebra $\mathbb{F}_5[H] \rtimes H$ has $(*)$.
\end{example}

\subsection{Translated regular orbits} 

The following synchronization result for linear groups
is the key tool for the proof of Theorem~\ref{thMain}.

\begin{proposition} \label{prop:translated}
Let $G$ be a finite nilpotent group and $V$ a completely reducible
faithful $G$-module.
Let $P_1,\ldots,P_n$ be the Sylow subgroups of $G=P_1\times\cdots\times P_n$,
and suppose that for each $i = 1,\ldots,n$ there exists $v_i \in V$ such that $C_{P_i}(v_i) = 1$.
Then for every $t_1,\ldots,t_n \in V$ there exists
$v\in V$ such that 
$$ C_{P_i}(v + t_i) \> = \> 1 $$
for every $i=1,\ldots,n$.
\end{proposition}

Note that $(|V|,|G|)=1$ follows by the nilpotence of $G$ and complete reducibility.
The situation where the $t_i$'s are zero vectors can be considered
a synchronization theorem for Sylow subgroups of nilpotent linear groups.
This special case is already nontrivial and implies the main theorem of Hargraves \cite{Har81} (assuming the result for linear $p$-groups).
The fact that the $t_i$'s can be chosen arbitrarily is surprising.

The proof of Proposition~\ref{prop:translated} is somewhat technical,
but the main idea is to use a tensor decomposition
to turn the stabilizer conditions into determinant conditions.
Suitable parametrizations make each determinant a polynomial of degree at most one.
The number of determinant conditions is smaller than the size of the relevant finite field,
which allows us to satisfy all the stabilizer conditions simultaneously.

\begin{proof}[Proof of Proposition~\ref{prop:translated}]
For each isotypic component $W$ of $V$,
let $v_{i,W},t_{i,W}\in W$ be the projections of $v_i$ and $t_i$ respectively.
We shall find $x_W \in W$ such that
\begin{equation} \label{eq:component-containment}
    C_{P_i}(x_W+t_{i,W}) \subseteq C_{P_i}(v_{i,W})
    \qquad (1\leqslant i\le n).
\end{equation}
For $v=\sum_W x_W$, we have
\[
    C_{P_i}(v+t_i)
    =\bigcap_W C_{P_i}\bigl(x_W+ t_{i,W} \bigr)
    \subseteq \bigcap_W C_{P_i}\bigl(v_{i,W}\bigr)
    =C_{P_i}(v_i)=1
\]
for every $i$, as desired.

Now let $F$ be the ground field and fix an isotypic component $W$ of $V$.
We suppress the subscript $W$, so let $v_1,\ldots,v_n,t_1,\ldots,t_n \in W$.
Write $W \cong S^{\oplus m}$, where $S$ is an irreducible $FG$-module,
and set $E=\mathrm{End}_{FG}(S)$.
By Schur's lemma and Wedderburn's theorem, $E$ is a finite field.
Its action gives $S$ an $E$-vector-space structure,
and the image of $FG$ on $S$ is $\mathrm{End}_E(S)$.
For each $i$, let $A_i \cong \mathrm{Mat}_{d_i}(E)$ be the $E$-algebra generated by the image of $P_i$ on $S$.
These algebras are semisimple, commute pairwise, and together generate $\mathrm{End}_E(S)$.
After omitting the $P_i$ acting trivially on $W$
and relabelling the remaining indices, we have
\[
    W\cong U_1\otimes_E\cdots\otimes_E U_k\otimes_E E^m,
\]
where $P_i$ acts only on $U_i$, absolutely irreducibly, and $\dim_E U_i=d_i$.

Let $q=|E|$.
The nontrivial image of $P_i$ on $U_i$ has a central element of order $p_i$,
which acts as a scalar by absolute irreducibility.
Hence $p_i\mid q-1$ for each $i\leqslant k$, and $k<q$ because these primes are distinct.

Choose bases in all tensor factors.
For $y \in W$, let $M_i(y)$ be its $i$th flattening:
the $d_i\times r_i$ matrix of the contraction map
\[
    \left(\bigotimes_{j\ne i}U_j\otimes_E E^m\right)^*
    \to U_i,
    \qquad
    r_i=m\prod_{j\ne i}d_j.
\]
An element of \(P_i\) fixes \(y\) if and only if it fixes every column of \(M_i(y)\). If \(M_i(y)\) has full row rank, then the elements of \(P_i\) fixing $y$ act trivially on \(U_i\) and hence on \(W\), so
\begin{equation}\label{eq:full-rank-centralizer}
    C_{P_i}(y) = C_{P_i}(W) .
\end{equation}
Since \(C_{P_i}(W)\subseteq C_{P_i}(v_i)\), \eqref{eq:full-rank-centralizer} gives \eqref{eq:component-containment} for $y = x +t_i$.
We shall use the elementary fact that a nonzero polynomial over $\mathbb{F}_q$ in $N$ variables, of degree less than $q$ in each variable, cannot vanish on all of $\mathbb{F}_q^N$.

\smallskip
\noindent\emph{Case 1: $d_i \leqslant r_i$ for every $i \leqslant k$.}
Let $X$ be a formal vector with indeterminates as the coordinates in the chosen tensor product basis.
For each $i$, choose a $d_i \times d_i$ minor of $M_i(X+t_i)$ and denote its determinant by $f_i(X)$.
Each $f_i$ is a nonzero polynomial of degree at most one in each coordinate.
Thus
\[
    \prod_{i=1}^k f_i(X)
\]
is nonzero and has degree at most $k<q$ in each coordinate.
Hence there exists $x\in W$ at which this product is nonzero,
and \eqref{eq:full-rank-centralizer} holds.

\smallskip
\noindent\emph{Case 2: $d_1>r_1$, after relabelling.}
There is at most one such index, since $r_i \geqslant d_j$ whenever $j\ne i$.
Put $R=U_2\otimes_E\cdots\otimes_E U_k\otimes_E E^m$, so that $W=U_1\otimes_E R$.
Choose a full-column-rank matrix
$A\in\mathrm{Mat}_{d_1\times r_1}(E)$ whose column space contains that of
$M_1(v_1)$. Let $B$ be an $r_1\times r_1$ matrix of independent
indeterminates, and define $x(B)$ by
\[
    M_1(x(B)+t_1)=AB.
\]
For every invertible specialization of $B$, the matrices $AB$ and
$A$ have the same column space. Hence
\begin{equation}\label{eq:first-factor}
    C_{P_1}(x(B)+t_1) \subseteq C_{P_1}(v_1).
\end{equation}

Let $y(B)$ be defined by $M_1(y(B))=AB$, and put $y=y(I)$.
Every $M_i(y)$ with $i>1$ has full row rank.
Indeed, otherwise some nonzero functional on $U_i$ would annihilate $y$ by contraction;
tensoring it with a nonzero functional on the remaining factors of $R$
would give a nonzero vector in the kernel
of the injective map $M_1(y) \colon R^* \to U_1$.
For each $i>1$, choose a $d_i\times d_i$ minor nonzero on $M_i(y)$,
and denote its determinant, evaluated on $M_i(x(B)+t_i)$, by $f_i(B)$.
Since
\[
    x(B)+t_i=y(B)+(t_i-t_1),
\]
the homogeneous part of degree $d_i$ of $f_i(B)$ is the corresponding minor of $M_i(y(B))$.

To bound the degree in an individual entry $B_{ab}$, write $A_a$ for column $a$ of $A$,
and let $e_b$ be the corresponding tensor product basis vector of $R$.
The coefficient of $B_{ab}$ in $y(B)$ is the pure tensor $A_a\otimes e_b$.
Its matrix in every flattening has rank at most one, so multilinearity of the determinant implies that each \(f_i(B)\) has degree at most one in \(B_{ab}\).
It follows that the product
\[
    \det(B) \prod_{i=2}^k f_i(B)
\]
is a nonzero polynomial of degree at most $k<q$ in each variable.
Choosing a specialization over $E$ at which it is nonzero,
\eqref{eq:first-factor} handles $P_1$,
while \eqref{eq:full-rank-centralizer} handles every $P_i$ with $i>1$.
Thus \eqref{eq:component-containment} holds in this case as well,
and the proof is complete.
\end{proof}

\subsection{Abstract groups}

The following observation explains how centralizers of vectors
are relevant in establishing Conjecture~\ref{conjMain}.

\begin{lemma}[Lemma~2.4 in \cite{LS26}] \label{lemIntCen} 
Let $G = V \rtimes K$.
For each $v \in V$, we have $C_K(v) =K \cap K^v$.
\end{lemma} 
\begin{proof}
It is easy to see that $C_K(v) \subseteq K \cap K^v$.
Now suppose $a = b^v \in K \cap K^v$ for some $a,b \in K$ and $v \in V$.
Multiplying by $b^{-1}$ we obtain $b^{-1}a = b^{-1} v^{-1} b v$.
The left side lies in $K$, while the right side lies in $V$, so $a=b \in C_K(v)$. 
\end{proof} 

We are ready for the proof of our main result.
We write $\pi(G)$ for the set of prime divisors of $|G|$.

\begin{proof}[Proof of Theorem~\ref{thMain}]
We work by induction on the order of $G$.
Let $P_i \in \Syl_{p_i}(G)$ for distinct primes $p_1,\ldots,p_n$,
so that we want to find $x \in G$ such that $P_i \cap P_i^x = O_{p_i}(G)$ for all $i$.

We first reduce to the situation where the Frattini subgroup $\Phi(G)$ is trivial.
This is because $\Phi(G/\Phi(G))=1$,
and $F(G/\Phi(G))=F(G)/\Phi(G)$ \cite[Theorem~10.6(c)]{DH92}.
In fact $\Phi(G)P_i/\Phi(G) \in \Syl_{p_i}(G/\Phi(G))$,
and working in $G/\Phi(G)$ we obtain $x \in G$
such that $\Phi(G)P_i \cap  \Phi(G)P_i^x \subseteq F(G)$,
which implies $P_i \cap P_i^x = O_{p_i}(G)$ for all $i$.

So $\Phi(G)=1$ and $F(G)$ is a direct product of elementary abelian groups.
Up to arranging the $p_i$'s, set $p=p_1$ such that $A=O_p(G) \neq 1$.

Observe that $G/A$ satisfies again the hypotheses of the theorem, and set $H/A = F(G/A)$.
By induction, there exists $xA \in G/A$ such that 
\begin{equation} \label{eq:Quoz1}
 P_i A \cap P_i^x A \subseteq H
\end{equation}
for every $i=1,\ldots,n$.
Since $H/A$ is a $p'$-group, we have $P_1 A \cap P_1^x A = A$.
If we set $Q_i=P_i \cap H$ then both $Q_i$ and $Q_i^x$ are Sylow $p_i$-subgroups of $H$.
Note that $AQ_i \unlhd H \unlhd G$, so $AQ_i$ has $(*)_{p_i}$ by Lemma~\ref{lem:normal}.
Moreover, by the Frattini argument we have $A N_H(Q_i)= H$,
so for each $i$ there exists $t_i \in A$ such that $Q_i^x=Q_i^{t_i}$.

Set $\overline{H}=H/C_H(A)$ and note that, by the discussion above and Lemma~\ref{lemIntCen},
the $\overline{H}$-module $A$ satisfies the hypotheses of Proposition~\ref{prop:translated}.
If $I=\{ 2\leqslant i\leqslant n : Q_i \nsubseteq C_H(A)\}$,
then $\pi(\overline{H})=\{ p_i \in \pi(G)\}_{i\in I}$
and $\{\overline{Q}_i=Q_iC_H(A)/C_H(A)\}_{i\in I}$ is a family of Sylow subgroups of $\overline{H}$.
We find $a \in A$ such that 
\[
\overline{Q}_i \cap C_{\overline{H}}(a+t_i) =1_{\overline{H}}
\]
for every $i\in I$.
If $i \notin I$ then $Q_i \subseteq C_H(A)$, so by Lemma~\ref{lemIntCen} we have
\begin{equation} \label{eq:Quoz2}
Q_i\cap Q_i^{xa}=Q_i\cap Q_i^{a t_i}=C_{Q_i}(a+t_i) \subseteq C_H(A) 
\end{equation}
for every $i\geqslant2$.

Note that $C_H(A)/A$ is nilpotent and $A$ is central in $C_H(A)$,
thus $C_H(A)$ is nilpotent and normal in $G$, so $C_H(A) \subseteq F(G)$.
Since $P_i \cap P_i^{xa} \subseteq H$ by (\ref{eq:Quoz1}), by (\ref{eq:Quoz2}) we conclude
that 
\[
P_i \cap P_i^{xa} =Q_i\cap Q_i^{xa} \subseteq Q_i \cap F(G) \subseteq O_{p_i}(G) 
\]
for every $i \geqslant 2$.
We also have $P_1 \cap P_1^{xa}=A$ since $xa \in xA$,
so $xa$ is the desired element.
\end{proof}

\section{Symmetric and alternating groups} \label{sec3}

In this section we prove that Conjecture~\ref{conjMain} holds
for all symmetric and alternating groups.
Let $G$ be $S_n$ or $A_n$ with $n \geqslant 5$.
Note that $O_p(G) = 1$ for every prime $p$, so $G$ has $(*)_p$
if and only if there exist two Sylow $p$-subgroups of $G$ intersecting trivially.
By \cite{MZ96}, $G$ has $(*)_p$ if and only if $(G,p) \ne (S_8,2)$.

Let us recall the probabilistic method used in \cite{BH26}
to establish Conjecture~\ref{conjMain} for non-alternating simple groups.
Suppose $G$ has $(*)$, let $G_p \in \mathrm{Syl}_p(G)$ and let
\begin{equation*}
    Q_p(G) = \frac{|\{x \in G : G_p\cap G_p^x \ne 1\}|}{|G|}
\end{equation*}
be the probability that $G_p\cap G_p^x \ne 1$ for a uniformly random chosen element $x\in G$.
By a union bound, Conjecture~\ref{conjMain} holds for $G$ if
\begin{equation*}
    \sum_{p\in\pi(G)} Q_p(G) < 1 .
\end{equation*}
Asymptotic estimates for $Q_p(G)$ in the cases of symmetric and alternating groups
are given in \cite{DGG+26,Ebe26}, but those bounds are ineffective for small $n$.
Instead, we proceed as in \cite[page~5]{BH26}.
Let $\{x_{p,1},\dots,x_{p,k}\}$ be a complete set of representatives
of the conjugacy classes in $G$ of elements of order $p$.
As observed by Liebeck and Shalev in \cite[Proof of Theorem~1.3]{LS99}, we have
\begin{equation*}
    Q_p(G) \leqslant\widehat{Q}_p(G):= \sum_{i = 1}^k \frac{|x_{p,i}^G\cap G_p|^2}{|x_{p,i}^G|},
\end{equation*}
noting that $|x_{p,i}^G\cap G_p|$ is independent from the choice of the Sylow $p$-subgroup $G_p$.
This immediately gives the following.

\begin{lemma}
    \label{l:sym_prob}
    Suppose $G$ has $(*)$. Then Conjecture~\ref{conjMain} holds for $G$ if
    \begin{equation*}
    \sum_{p\in\pi(G)} \widehat{Q}_p(G) < 1.
\end{equation*}
\end{lemma} 

For $n \leqslant 40$,
the validity of Conjecture~\ref{conjMain} for symmetric and alternating groups
can be checked using {\sc Magma} (see also \cite[Remark 5.3]{BH26}).
The verification code is available in \cite{HLRS26}.
Therefore, we can assume $n > 40$.

\subsection{Symmetric groups}

We first deal with the symmetric groups.
We may assume that $x_{p,j}$ has cycle type $[p^j,1^{n-pj}]$, so
\begin{equation*}
    |x_{p,j}^{S_n}| = \frac{n!}{p^j j! (n-pj)!} =: C(n,p,j).
\end{equation*}
Defining $a(n,p,j) := |x_{p,j}^{S_n} \cap (S_n)_p|$, we have
\begin{equation*}
    \widehat{Q}_p(S_n) = \sum_{j = 1}^{\lfloor n/p \rfloor} \frac{a(n,p,j)^2}{C(n,p,j)}.
\end{equation*}

Our goal is to give a good bound for $a(n,p,j)$.
For integers $m\geqslant1$ and $j\geqslant0$, let
\begin{equation*}
    B(p,m,j) := (p-1)^j{m+j-1\choose j}.
\end{equation*}

\begin{lemma}\label{l:sym_a<B}
    For every prime $p$, every $n$ and every $j\geqslant 0$, we have
    \begin{equation*}
        a(n,p,j) \leqslant B\left( p, \left\lfloor\frac{n}{p}\right\rfloor, j\right).
    \end{equation*}
\end{lemma}

\begin{proof}
    We first consider the case where $n = p^r$. Set $m = p^{r-1}$, and define
    \begin{equation*}
        F_r(X) = \sum_{j\geqslant 0} \frac{a(p^r,p,j)}{(p-1)^j}X^j.
    \end{equation*}
    A standard Sylow $p$-subgroup $P_r$ of $S_n$ is given by
    \begin{equation*}
        P_1 = C_p, \quad P_k = P_{k-1}\wr C_p.
    \end{equation*}
    We have $F_1(X) = 1+X$.
    For $r\geqslant 2$, we claim by induction that
    \begin{equation}\label{e:sym_F}
        F_r(X) = F_{r-1}(X)^p+\left( \frac{p}{p-1} \right)^{m-1}X^{m}.
    \end{equation}
    To see this, first note that elements $g$ satisfying $g^p = 1$
    in the base group $P_{r-1}$ contribute $F_{r-1}(X)^p$.
    For each of the $p-1$ non-identity elements of the top group $C_p$,
    the condition $g^p = 1$ determines one base coordinate from the other $p-1$.
    This implies that there are
    \begin{equation*}
        (p-1)|P_{r-1}|^{p-1} = (p-1)p^{m-1}
    \end{equation*}
    such elements.
    Each of these elements is fixed-point-free and hence has exactly $m$ cycles of length $p$. Normalizing by $(p-1)^{m}$ gives \eqref{e:sym_F}.

    We now show that
    \begin{equation*}
        [X^m]F_r(X) \leqslant 3^{m-1}.
    \end{equation*}
 Indeed, since $F_1(X) = 1+X$, this is clearly true for $r = 1$. For $r\geqslant 2$ we argue by induction, so by \eqref{e:sym_F}, we have
    \begin{equation*}
        [X^m]F_r(X) = ([X^{m/p}]F_{r-1}(X))^p+\left( \frac{p}{p-1} \right)^{m-1} \leqslant 3^{m-p}+2^{m-1} \leqslant3^{m-1}
    \end{equation*}
    as required. Moreover, one sees that
    \begin{equation*}
        [X^m]F_r(X) \leqslant 3^{m-1}\leqslant {2m-1\choose m} = [X^m](1-X)^{-m}.
    \end{equation*}
We now prove the coefficientwise bound by induction on $r$.
The case $r=1$ follows from
$F_1(X)=1+X\le (1-X)^{-1}$.
For $r\ge 2$, the induction hypothesis gives
\[
F_{r-1}(X)\le (1-X)^{-m/p}
\]
coefficientwise. Since all coefficients are nonnegative,
we may raise this inequality to the $p$th power.
For $0\le \ell<m$, the additional term in the recurrence
contributes nothing, so
\[
[X^\ell]F_r(X)
=
[X^\ell]\bigl(F_{r-1}(X)^p\bigr)
\le
[X^\ell](1-X)^{-m}.
\]
Together with the estimate for the coefficient of $X^m$
above and the fact that $\deg F_r=m$, this yields
    \begin{equation}\label{e:sym_F_2}
        F_r(X) \leqslant (1-X)^{-m}
    \end{equation}
    coefficientwise. That is, $[X^k]F_r(X) \leqslant[X^k](1-X)^{-m}$ for all $k\geqslant 0$.

    Finally, we consider the general case with $n = d_0+d_1p +\cdots +d_Rp^R$ for $0\leqslant d_r < p$. Then a Sylow $p$-subgroup of $S_n$ is a direct product of $d_r$ copies
    of $P_r$ acting on disjoint sets.
    Then by \eqref{e:sym_F_2},
    \begin{equation*}
        \sum_{j\geqslant 0}\frac{a(n,p,j)}{(p-1)^j}X^j = \prod_{r = 1}^RF_r(X)^{d_r} \leqslant (1-X)^{-\sum_{r=1}^Rd_rp^{r-1}} = (1-X)^{-\lfloor n/p\rfloor}.
    \end{equation*}
    Taking the coefficient of $X^j$ completes the proof.
\end{proof}

Now we define
\begin{equation*}
    M(n,p) := \sum_{j = 1}^{\lfloor n/p \rfloor} \frac{B(p,\lfloor n/p \rfloor,j)^2}{C(n,p,j)} ,
\end{equation*}
so that $\widehat{Q}_p(S_n) \leqslant M(n,p)$ by Lemma~\ref{l:sym_a<B}.
The following result is purely numerical and we omit its proof.

\begin{lemma}\label{l:sym_m_n}
    For $n>40$, any prime $p$ and integer $n_0$ with $p \mid n_0 \leqslant n$,
    we have $M(n,p) \leqslant M(n_0,p)$.
    In particular, $M(n,p) \leqslant M(p,p)$.
\end{lemma}

We are ready to establish Conjecture~\ref{conjMain} for symmetric groups.

\begin{theorem}\label{t:sym}
    For $n > 40$, we have
    \begin{equation*}
        \sum_{p\in\pi(S_n)} \widehat{Q}_p(S_n) < 1.
    \end{equation*}
\end{theorem}

\begin{proof}
    By Lemma~\ref{l:sym_a<B} we have $\widehat{Q}_p(S_n) \leqslant M(n,p)$, so
    \begin{equation*}
        \begin{aligned}
            \widehat{Q}_2(S_n) \leqslant M(n,2) &\leqslant M(40,2) < \frac{3}{4} , \\
            \widehat{Q}_3(S_n) \leqslant M(n,3) &\leqslant M(39,3) < \frac{1}{25} , \\
            \widehat{Q}_5(S_n) \leqslant M(n,5) &\leqslant M(40,5) < \frac{1}{10000} , \\
            \widehat{Q}_7(S_n) \leqslant M(n,7) &\leqslant M(35,7) < \frac{1}{2400} ,
        \end{aligned}
    \end{equation*}
    by direct computation and appealing to Lemma~\ref{l:sym_m_n}.
    Moreover, one can show by induction that for any integer $r \geqslant 10$, we have $2^{r+1}r^2 \leqslant r!$, which yields
    \begin{equation*}
        \widehat{Q}_p(S_n) \leqslant M(p,p) = \frac{(p-1)^2}{(p-1)!} \leqslant2^{-p}
    \end{equation*}
    for $p \geqslant 11$, by taking $r = p-1$. This gives
    \begin{equation*}
        \sum_{p \geqslant11}\widehat{Q}_p(S_n) \leqslant \sum_{k = 11}^{\infty} 2^{-k} = \frac{1}{1024}.
    \end{equation*}
    Therefore,
    \begin{equation*}
        \sum_{p\in\pi(S_n)} \widehat{Q}_p(S_n) \leqslant \frac{3}{4} + \frac{1}{25} + \frac{1}{10000} +\frac{1}{2400} + \frac{1}{1024} < 1 . \qedhere
    \end{equation*}
\end{proof}

\subsection{Alternating groups}

To deal with the involutions in $A_n$,
we consider the contribution of even permutations to $\widehat{Q}_2(S_n)$.
That is,
\begin{equation*}
    \widehat{Q}_{\mathrm{ev}}(S_n) = \sum_{\substack{1\leqslant j \leqslant\lfloor n/2\rfloor\\ \mbox{$j$ even}}} \frac{|x_{2,j}^{S_n}\cap (S_n)_2|^2}{|x_{2,j}^{S_n}|}
    =
    \sum_{\substack{1\leqslant j \leqslant\lfloor n/2\rfloor\\ \mbox{$j$ even}}} 
    \frac{a(n,2,j)^2}{C(n,2,j)}.
\end{equation*}

\begin{lemma}\label{l:alt_4}
    We have
    \begin{equation*}
        \widehat{Q}_p(A_n) \leqslant
        \begin{cases}
            4\widehat{Q}_{\mathrm{ev}}(S_n), \quad &p = 2;\\
            4\widehat{Q}_p(S_n),\quad & \mbox{$p$ odd.}
        \end{cases}
    \end{equation*}
\end{lemma}

\begin{proof}
    For $p\geqslant 3$, let $\mathcal{P}_p$ be the set of elements in $S_n$ of order $p$, and we set $\mathcal{P}_2$ to be the set of involutions in $S_n$ that are even permutations. So for every $p$, $\mathcal{P}_p$ is a subset of $A_n$. Since
    \begin{equation*}
        |x^{S_n}\cap (S_n)_p| \geqslant |x^{S_n} \cap (A_n)_p| \geqslant |x^{A_n} \cap (A_n)_p|
    \end{equation*}
    and $|x^{S_n}| \leqslant 2|x^{A_n}|$, we see that
    \begin{equation*}
        \frac{1}{4}\widehat{Q}_p(A_n) = \sum_{x\in\mathcal{P}_p} \frac{|x^{A_n}\cap (A_n)_p|^2}{(2|x^{A_n}|)^2} \leqslant \sum_{x\in\mathcal{P}_p} \frac{|x^{S_n}\cap (S_n)_p|^2}{|x^{S_n}|^2} = 
        \begin{cases}
            \widehat{Q}_{\mathrm{ev}}(S_n), \quad &p = 2;\\
            \widehat{Q}_p(S_n),\quad & \mbox{$p$ odd.}
        \end{cases}
    \end{equation*}
    This completes the proof.
\end{proof}

For odd $p$, we can bound $\widehat{Q}_p(A_n) \leqslant 4\widehat{Q}_p(S_n) \leqslant 4M(n,p)$ by Lemma~\ref{l:sym_a<B}. Similarly, we have $\widehat{Q}_2(A_n)\leqslant 4M_{\mathrm{ev}}(n)$ with
\begin{equation*}
    M_{\mathrm{ev}}(n):= \sum_{\substack{1\leqslant j \leqslant\lfloor n/2\rfloor\\ \mbox{$j$ even}}} \frac{B(2,\lfloor n/2 \rfloor,j)^2}{C(n,2,j)}.
\end{equation*}

We obtain a purely numerical result for $M_{\mathrm{ev}}$; we omit the proof.

\begin{lemma}\label{l:alt_ev}
    For any integer $n \geq 40$,
    we have $M_{\mathrm{ev}}(n) \leqslant M_{\mathrm{ev}}(40) <\frac{3}{16}$. 
\end{lemma}

\begin{theorem}\label{t:alt}
    For $n > 40$, we have
    \begin{equation*}
        \sum_{p \in \pi (A_n)} \widehat{Q}_p(A_n) < 1.
    \end{equation*}
\end{theorem}

\begin{proof}
    With Lemmas~\ref{l:alt_4} and \ref{l:alt_ev} in hand,
    the proof is almost identical to that of Theorem~\ref{t:sym}.
    More precisely, we have
    \begin{equation*}
        \begin{aligned}
            &\widehat{Q}_2(A_n) \leqslant 4\widehat{Q}_{\mathrm{ev}}(S_n) \leqslant 4M_{\mathrm{ev}}(n) \leqslant 4M_{\mathrm{ev}}(40) < \frac{3}{4} , \\
            &\widehat{Q}_3(A_n) \leqslant 4\widehat{Q}_3(S_n) \leqslant 4M(n,3) \leqslant 4M(39,3) < \frac{4}{25} , \\
            &\widehat{Q}_5(A_n) \leqslant 4\widehat{Q}_5(S_n) \leqslant 4M(n,5) \leqslant 4M(40,5) < \frac{1}{2500} , \\
            &\widehat{Q}_7(A_n) \leqslant 4\widehat{Q}_7(S_n) \leqslant 4M(n,7) \leqslant 4M(35,7) < \frac{1}{600} ,\\
            &\sum_{p\geqslant 11}\widehat{Q}_p(A_n) \leqslant 4\sum_{p\geqslant 11}\widehat{Q}_p(S_n) \leqslant \frac{1}{256},
        \end{aligned}
    \end{equation*}
    which yields
    \begin{equation*}
        \sum_{p\in\pi(A_n)} \widehat{Q}_p(A_n) \leqslant \frac{3}{4} + \frac{4}{25} + \frac{1}{2500} +\frac{1}{600} + \frac{1}{256} < 1. \qedhere
    \end{equation*}
\end{proof}

\begin{remark}
Arguing as in \cite[Theorem~F]{BH26},
Theorems~\ref{t:sym} and \ref{t:alt} provide an independent proof of the main result in \cite{Zen14},
which asserts the existence of $x \in G$ such that $A \cap B^x = 1$,
for any pair $(A,B)$ of nilpotent subgroups of $G\in\{S_n,A_n\}$ with $n \geqslant 9$
(the cases $9 \leqslant n \leqslant 40$ can be checked in {\sc Magma} using the code in \cite{HLRS26}).
\end{remark}

\vspace{0.2cm} \noindent
 {\bfseries Acknowledgments.}
HYH was supported by the National Research, Development and Innovation Office (NKFIH) Grant No. K153681. HYH and AR would like to thank the Isaac Newton Institute for Mathematical Sciences, Cambridge, for support and hospitality during the programme \textit{Algebraic groups, geometry, invariants and related topics} (EPSRC grant EP/
Z000580/1), where part of work on this paper was undertaken. AR also acknowledges support from the Additional Funding Programme for Mathematical Sciences, delivered by EPSRC (EP/V521917/1), and the Heilbronn Institute for Mathematical Research.

\vspace{0.2cm} \noindent
 {\bfseries AI statement:}
 AI was used during the preparation of this manuscript.
 In particular, a version of Proposition~\ref{prop:translated}
 was discovered independently by HYH and AR, and by FL and LS, while using OpenAI's GPT.
 Both teams were trying to establish the odd order case.
 Similarly, we have been significantly helped with the results in Section~\ref{sec3}.
 All our results have been formalized in Lean~4 using mathlib, and the formalization is available in~\cite{HLRS26}.
 We take full responsibility for the content of the paper.

   \vspace{0.2cm}

\end{document}